\documentclass{amsart}
\usepackage{amsaddr}
\usepackage{array,geometry,amsmath,amsthm,amssymb,graphicx,tikz,tikz-cd,mathrsfs,fancyhdr,listings}
\usepackage[thinlines]{easytable}
\usepackage{bigstrut}
\usepackage{perpage}
\MakePerPage{footnote}

\usepackage{hyperref}
\hypersetup{colorlinks=true,citecolor=purple,linkcolor=purple}
\usepackage{ZJSstyle}

\title{Short paths in PU(2)}
\author{Zachary Stier}
\email{zstier@berkeley.edu}
\address{UC Berkeley}
\date{December 2020}

\definecolor{codegreen}{rgb}{0,0.6,0}
\definecolor{codegray}{rgb}{0.5,0.5,0.5}
\definecolor{codepurple}{rgb}{0.58,0,0.82}
\definecolor{codeblue}{rgb}{0,0,0.95}
\definecolor{backcolor}{rgb}{0.95,0.95,0.92}
\begin{document}

\maketitle 

\begin{abstract}
Parzanchevski--Sarnak recently adapted an algorithm of Ross--Selinger for factorization of $\PU(2)$-diagonal elements to within distance $\eps$ into an efficient probabilistic algorithm for any $\PU(2)$-element, using at most $3\log_p\frac{1}{\eps^3}$ factors from certain well-chosen sets. The Clifford+$T$ gates are one such set arising from $p=2$. In that setting, we leverage recent work of Carvalho Pinto--Petit to improve this to $\frac{7}{3}\log_2\frac{1}{\eps^3}$, and implement the algorithm in Haskell. 
\end{abstract}

\section{Introduction}
	Factoring in a matrix group given a set of topological generators is a problem of fundamental and practical significance. Because individual gates in quantum circuits acting on $n$ qubits exist as (projective) elements in $U\lpr{\C^{2^n}}$, a question of importance is to fabricate only a specific (finite) set which may then approximate any given gate to arbitrary precision; that this is possible is the Solovay--Kitaev theorem (cf.~Sarnak \cite{Sar15}). A golden gate set \cite{Sar15,PS} is a certain choice of these topological generators for single-qubit gates, and their construction is associated to a particular prime $p$. \cite{PS} shows that most elements of $\PU(2)$ will have a factorization to within $\eps$ (in the bi-invariant metric $d$ on $\PU(2)$) of length up to $k=(1+o(1))\log_p\frac{1}{\eps^3}$,\footnote{Which is optimal due to the set of such gates having size $\approx p^k$, and the ball of radius $\eps$ in $\PU(2)$ having volume (in Haar measure) some constant multiple of $\eps^3$.} but that computing such factorizations for general elements is NP-complete. Instead, \cite[\S2.3]{PS} extends \cite{RS16}'s $(3+o(1))\log_2\frac{1}{\eps^3}$ efficient heuristic factorization algorithm for Clifford+$T$ gates, to any golden gate set. The bulk of this factorization is accomplished by thrice applying an algorithm adapted from \cite{RS16} (initially written for Clifford+$T$ gates, defined below), each time factoring a diagonal into up to length $(1+o(1))\log_p\frac{1}{\eps^3}$ in topological generators. 
	
	\cite{CPP} point out that their factorization for the discrete case of the LPS Ramanujan graphs $X^{p,q}$, attaining factorization length $7\log_pq$, should analogize in the continuous setting with $q\approx\frac{1}{\eps}$, and indeed it does, as we now show for the setting of Clifford+$T$ gates. We understand that Kliuchnikov--Lauter--Minko--Paetznick--Petit have a similar extension underway \cite{KLM+}. 
	
	In $\PU(2)$, let
	\begin{align*}
		H=\frac{i}{\sqrt{2}}\begin{pmatrix}1&\phantom{-}1\\1&-1\end{pmatrix}, && S=\begin{pmatrix}1\\&i\end{pmatrix}, && T=\begin{pmatrix}e^{\frac{i\pi}{8}} \\ & e^{-\frac{i\pi}{8}}\end{pmatrix}.
	\end{align*}
	Associated to the prime $p=2$ is a golden gate set $\cS$ consisting of the Clifford group $C=\<H,S\>$ along with $T$. Define $\G=\<\cS\>=\<C\sqcup\{T\}\>=\<H,T\>$; for this reason we often think of $\cS$ as just $\{H,T\}$. $\G$ is dense in $\PU(2)$, and we are interested in tracking the $T$-count of $\G$-approximants to elements of $\PU(2)$, defined as follows: for $\g\in\G$, factor $\g=\prod\limits_{i=1}^ns_i$ where $s_i\in\cS$; the $T$-count is the minimum number of $s_i$ equalling $T$ among all factorizations $\{s_i\}$ of $\g$. The $T$-count is of interest because the Clifford+$T$ gates are commonly used for modern quantum circuitry, and elements of $C$ are assumed to be easy to fabricate, while $T$ is much more costly, hence why we ignore the number of $C$-elements required. 
	
	Each element of $\G$ has entries in the standard quaternion realization \cite{Sar15} in the ring of integers $\cO=\Z[\sqrt{2}]$, a unique factorization domain, with fraction field $K=\Q[\sqrt{2}]$. We assume that there exists a $O(\poly\log\tN_{K/\Q}(n))$-time algorithm to factor $n\in\cO$ into primes. 
	\begin{theorem-non}[{\cite{RS16}}]
		There is a $O(\poly\log\frac{1}{\eps})$-time algorithm to factor almost any element of $\PU(2)$ using the gates $\cS$ to within $\eps$ in $d$ using $T$-count at most $(3+o(1))\log_2\frac{1}{\eps^3}$. 
	\end{theorem-non}
	(This result, accomplished with Euler angles, generalizes to when $\cS$ is any golden gate set \cite[\S2.3]{PS}.) We shall leverage the two-dimensional lattice structure of $\cO$ to apply Lenstra's algorithm \cite{Len,Paz}. The main result is the following: 
	\begin{theorem}\label{thm:main}
		There is a $O\lpr{\poly\log\frac{1}{\eps}}$-time algorithm to factor almost any element of $\PU(2)$ using the gates $\cS$ to within $\eps$ in $d$ using $T$-count at most $\lpr{\frac{7}{3}+o(1)}\log_2\frac{1}{\eps^3}$. 
	\end{theorem}
	The heart of this improvement lies with replacing the middle diagonal factor of a typical element of $\PU(2)$ with a well-chosen element of $\G$, selected to have particularly small $T$-count. 
	
	In \secref{sec:nearby}, we present some technical lemmas that enable this factorization. In \secref{sec:continuous alg}, we describe the algorithm which accomplishes \thmref{thm:main}, analyzing it is \secref{sec:analysis}. In \secref{sec:implementation}, we give some information about the Haskell implementation. 
	
\section{Nearby elements in $\PU(2)$}\label{sec:nearby}
	The contents of this section are entirely independent of the choice of topological generators. Instead, the goal is \lemref{lem:close approx} by which we establish sufficient conditions for two elements of $\PU(2)$ to be nearby with respect to the following metric. 
	\begin{definition-non}[$\PU(2)$'s bi-invariant metric, cf.~\cite{Sar15,PS}]
		$\PU(2)$ has the bi-invariant metric 
		$$d(x,y)=1-\half\abs{\tr x^*y}.$$
	\end{definition-non}
	Note that this is also the bi-invariant metric on $\SU(2)$, and that $\PU(2)\cong\PSU(2)$, so quotienting by (norm-one) scalars does not change this metric due to the presence of the absolute value. Therefore, for concreteness we shall work work elements of $\SU(2)$. 
	
	We introduce the following convenient notation for elements of the group. 
	\begin{notation-non}
		Suppose $\ga,\gb\in\C$ satisfy $\abs{\ga}^2+\abs{\gb}^2=1$. Then we let $u(\ga,\gb)$ denote the canonical element corresponding to $\ga$ and $\gb$, i.e.
		$$u(\ga,\gb)=\begin{pmatrix}\ga & \gb \\ -\ol{\gb} & \ol{\ga}\end{pmatrix}.$$
		We use the additional shorthand of $u(\gt)$ to denote the canonical diagonal element corresponding to rotation by $\gt\in\R$, i.e.
		$$u(\gt)=u(e^{i\gt},0)=\begin{pmatrix}e^{i\gt} & 0 \\ 0 & e^{-i\gt}\end{pmatrix}.$$
	\end{notation-non}
	Now, we are equipped to begin discussing the elements of $\SU(2)$, beginning with a fact about real numbers, representing the norms of the complex numbers comprising the matrix. There is a catch, namely that these results only hold for matrices outside of a small neighborhood of the identity matrix parameterized by $\eps_0$; however, $\eps_0$ can be made as small as one likes, with the only cost being (approximately inversely) to multiples of $\eps$ obtained in subsequent bounds. Therefore it is ideal to have $\eps\ll\eps_0$. 
	\begin{lemma}\label{lem:close norms}
		Select absolute constants $\tilde{\eps},\eps_0>0$. Suppose $a_1,a_2,b_1,b_2\ge0$ satisfy the following conditions, for some $\eps<\tilde{\eps}$: 
		\begin{align*}
			a_1^2+b_1^2=1, && a_2^2+b_2^2=1, && \abs{a_1-a_2}<\eps,
		\end{align*}
		and either $a_1<\sqrt{1-\eps_0^2}$ or $a_2<\sqrt{1-\eps_0^2}$. Then there exists $c=c(\tilde{\eps},\eps_0)$ for which $\abs{b_1-b_2}<c\eps$. 
	\end{lemma}
	A pictorial view of this result is as follows. Consider points in the upper half-plane and on the unit circle, at least some fixed distance above the $x$-axis. Then, if they have nearby $x$-coordinates, it follows that they have comparably nearby $y$-coordinates. 
	\begin{proof}
		For convenience, write $a_2=a_1+\gd$ for $\abs{\gd}<\eps$. Then 
		\begin{align*}
			b_2^2&=1-a_2^2\\
				&=1-a_1^2-2\gd a_1-\gd^2\\
				&=b_1^2-\gd(2a_1+\gd)\\
			\abs{b_1-b_2}&=\frac{\abs{\gd}\abs{2a_1+\gd}}{b_1+b_2}\\
				&<\frac{2+\tilde{\eps}}{\eps_0}\eps
		\end{align*}
		where $c(\tilde{\eps},\eps_0)=\frac{2+\tilde{\eps}}{\eps_0}$. 
	\end{proof}
	We leverage this result and apply it to the absolute values of entries of elements of $\SU(2)$, to show the following about such matrices having nearby entrywise absolute values and arguments. 
	
	\begin{lemma}\label{lem:close metric}
		Select absolute constants $\tilde{\eps},\eps_0>0$. Suppose $\ga_1,\ga_2,\gb_1,\gb_2\in\C$ satisfy the following conditions, for some $\eps<\tilde{\eps}$: 
		\begin{align*}
			\arg\ga_1&=\arg\ga_2, &&& \abs{\ga_1}^2+\abs{\gb_1}^2&=1,\\
			\arg\gb_1&=\arg\gb_2, &&& \abs{\ga_2}^2+\abs{\gb_2}^2&=1,\\
			&&\abs{\abs{\ga_1}-\abs{\ga_2}}&<\eps, 
		\end{align*}
		and either $\abs{\ga_1}<\sqrt{1-\eps_0^2}$ or $\abs{\ga_2}<\sqrt{1-\eps_0^2}$. Write $\g_1=u(\ga_1,\gb_1)$ and $\g_2=u(\ga_2,\gb_2)$. Then there exists $\tilde{c}=\tilde{c}(\tilde{\eps},\eps_0)$ for which $d(\g_1,\g_2)<\tilde{c}\eps$. 
	\end{lemma}
	\begin{proof}
		We first fix $c=\frac{2+\tilde{\eps}}{\eps_0}$ and immediately we apply \lemref{lem:close norms} to find $\abs{\abs{\gb_1}-\abs{\gb_2}}<c\eps$. We have that 
		$$\g_1^*\g_2=\begin{pmatrix}\ol{\ga_1}\ga_2+\gb_1\ol{\gb_2} & \star \\ \star & \ga_1\ol{\ga_2}+\ol{\gb_1}\gb_2\end{pmatrix}$$
		where the $\star$'s indicate irrelevant quantities. Then, 
		\begin{align*}
			\tr\g_1^*\g_2&=\ol{\ga_1}\ga_2+\ga_1\ol{\ga_2}+\gb_1\ol{\gb_2}+\ol{\gb_1}\gb_2\\
				&=2\abs{\ga_1}\abs{\ga_2}+2\abs{\gb_1}\abs{\gb_2}\\
				&>2\abs{\ga_1}\lpr{\abs{\ga_1}-\eps}+2\abs{\gb_1}\lpr{\abs{\gb_1}-c\eps}\\
				&=2-2\lpr{\abs{\ga_1}+c\abs{\gb_1}}\eps.
		\end{align*}
		Elementarily,\footnote{e.g.~using Lagrange multipliers.} $\abs{\ga_1}+c\abs{\gb_1}\le\sqrt{1+c^2}$ for all choices of $\ga_1,\gb_1$ satisfying $\abs{\ga_1}^2+\abs{\gb_1}^2=1$. Therefore, $\tr\g_1^*\g_2>2-2\sqrt{1+c^2}\eps$. We conclude with
		$$d(\g_1,\g_2)=1-\half\abs{\tr\g_1^*\g_2}<\sqrt{1+c^2}\eps$$
		where $\tilde{c}(\tilde{\eps},\eps_0)=\sqrt{1+c^2}$. 
	\end{proof}
	Finally, we use this guarantee about the bi-invariant metric to explicitly describe rotations which send one matrix to have close entries in absolute value to another while also ensuring closeness in the metric. 
	\begin{lemma}\label{lem:close approx}
		Select absolute constants $\tilde{\eps},\eps_0>0$. Take $\g_1,\g_2\in \SU(2)$ and write them as $\g_1=u(\ga_1,\gb_1)$ and $\g_2=u(\ga_2,\gb_2)$. If $\abs{\abs{\ga_1}-\abs{\ga_2}}<\eps$ for some $\eps<\tilde{\eps}$ and either $\abs{\ga_1}<\sqrt{1-\eps_0^2}$ or $\abs{\ga_2}<\sqrt{1-\eps_0^2}$ then there exist $\gt_1,\gt_2\in\R$ and $\tilde{c}=\tilde{c}(\tilde{\eps},\eps_0)$\footnote{$\tilde{c}$ as in \lemref{lem:close metric}.} for which, writing $\gd_1=u(\gt_1)$ and $\gd_2=u(\gt_2)$, we have 
		$$d(\g_1,\gd_1\g_2\gd_2)<\tilde{c}\eps.$$
	\end{lemma}
	\begin{proof}
		First, let
		\begin{align*}
			\gt_1+\gt_2&=\arg\ga_1-\arg\ga_2 \\
			\gt_1-\gt_2&=\arg\gb_1-\arg\gb_2
		\end{align*}
		which reduces to
		\begin{align*}
			\gt_1&=\half\lpr{\arg\ga_1-\arg\ga_2+\arg\gb_1-\arg\gb_2} \\
			\gt_2&=\half\lpr{\arg\ga_1-\arg\ga_2-\arg\gb_1+\arg\gb_2}.
		\end{align*}
		We then multiply out $\gd_1\g_2\gd_2$ as\footnote{The bars represent complex conjugates, not fractions.}
		$$\gd_1\g_2\gd_2=\begin{pmatrix}e^{i(\gt_1+\gt_2)}\ga_2 & e^{i(\gt_1-\gt_2)}\gb_2 \\ -\ol{e^{i(\gt_1-\gt_2)}\gb_2} & \ol{e^{i(\gt_1+\gt_2)}\ga_2}\end{pmatrix}=\begin{pmatrix}\ga_3 & \gb_3 \\ -\ol{\gb_3} & \ol{\ga_3}\end{pmatrix}\in \PU(2)$$
		where $\abs{\ga_3}=\abs{\ga_2}$, $\abs{\gb_3}=\abs{\gb_2}$, $\arg\ga_3=\arg\ga_1$, and $\arg\gb_3=\arg\gb_1$. Therefore we apply \lemref{lem:close metric} and immediately conclude the result. 
	\end{proof}
	
\section{Algorithm for short paths}\label{sec:continuous alg}
	Select absolute constants $\tilde{\eps},\eps_0>0$ where $\tilde{\eps}<\half$. Take any $g=u(\ga,\gb)\in \PU(2)$ where $\abs{\ga}<\sqrt{1-\eps_0^2}$, and pick $\eps<\tilde{\eps}$.\footnote{By choice of sufficiently small $\eps_0$, we exclude at this stage an arbitrarily small ball in Haar measure about the identity matrix.} We wish to approximate $g$ using $\g\in\G$ of the form
	\begin{equation}
		\g=\frac{1}{2^{\half k}}\begin{pmatrix}x_0+x_1i & x_2+x_3i \\ -x_2+x_3i & x_0-x_1i\end{pmatrix}\label{eq:gamma def}
	\end{equation}
	having $k$, the factorization length, minimized, and so we begin with $k=0$. (We also have $x_0,x_1,x_2,x_3\in\cO$.) In particular, the objective is to approximate $g$ as $\g_1\g\g_2$ where $\g_1,\g_2\in\G$ approximate well-chosen diagonals, and $\g\in\G$ has factorization computable by Kliuchnikov--Maslov--Mosca \cite{KMM}. (We will compute $\g_1$ and $\g_2$ using \cite{RS16}.) We will see that $\g$ is designed to have factorization typically shorter than that of $\g_1$ and $\g_2$, giving rise to the desired improvement. 
	
	In order to apply \lemref{lem:close approx} we need to have $\abs{\frac{x_0+x_1i}{2^{\half k}}}=\sqrt{\frac{x_0^2+x_1^2}{2^k}}$ near $\abs{\ga}$ (that is, within $\eps$). Because $\abs{\frac{x_0+x_1i}{2^{\half k}}}+\abs{\ga}\ge\abs{\ga}$ which is fixed, it suffices to first find candidate values for $x_0,x_1\in\cO$ with $\abs{\abs{\frac{x_0+x_1i}{2^{\half k}}}^2-\abs{\ga}^2}<\eps\abs{\ga}$, rewritten to
	\begin{equation}
		\abs{x_0^2+x_1^2-\abs{\ga}^22^k}<\eps\abs{\ga}2^k.\label{eq:first xs}
	\end{equation}
	Viewing $\g$ as an element of $\SU(2)$, we also have $\det\g=1$, i.e.~$x_0^2+x_1^2+x_2^2+x_3^2=2^k$. $[K:\Q]=2$, so we explicitly work with the Galois group elements
	\begin{align*}
		\gs_+:\hspace{0.18cm}1&\longmapsto1\\
		\sqrt{2}&\longmapsto\sqrt{2},\\
		\gs_-:\hspace{0.18cm}1&\longmapsto1\\
		\sqrt{2}&\longmapsto-\sqrt{2},
	\end{align*}
	both of which are real embeddings, so that as $x_i\in\cO\subset K\subset\R$, it follows that $\gs_\pm(x_0^2+x_1^2)+\gs_\pm(x_2^2+x_3^2)=\gs_\pm 2^k=2^k$, and so
	\begin{equation}
		\gs_\pm(x_0^2+x_1^2)\le 2^k.\label{eq:second xs}
	\end{equation}
	Now, let $m=x_0^2+x_1^2\in\cO$. Considering $\cO$ as an integer lattice, we adapt \eqref{eq:first xs} and \eqref{eq:second xs} and seek to solve
	\begin{align*}
		\abs{m-\abs{\ga}^22^k}&<\eps\abs{\ga}2^k\\
		\abs{\gs_\pm m}&\le 2^k
	\end{align*}
	which are convex constraints on $m$ when written in its lattice components. Since this is an integer programming problem in two dimensions, we apply Lenstra's algorithm to efficiently list all such lattice points $m$. For each $m$, using efficient factorization in $\cO$, we attempt to write $m$ as a sum of two squares; if possible, say $m=x_0^2+x_1^2$, and so we attempt to write $\nye{m}=2^k-m$ as a sum of two squares. If possible, say $\nye{m}=x_2^2+x_3^2$, so we simply halt and return $\g$ corresponding to \eqref{eq:gamma def}. However, if $\nye{m}$ may not be represented as a sum of two squares, we simply move on to the next value of $m$ and try this process again. If this fails for all $m$ arising from $k$, we increment $k$ and run Lenstra's algorithm for the new inequalities. 
	
	Supposing we have halted and constructed $\g$, we compute $\gd_1$ and $\gd_2$ guaranteed by \lemref{lem:close approx}. These are efficiently approximable by \cite[Algorithm 7.3]{RS16} to $\g_1$ and $\g_2$, respectively. Chaining together the three approximations $\g_1\g\g_2$ gives the final desired approximation to $g$. 
	
	\begin{remark-non}
		We assume to begin that $g$ is far from a diagonal, i.e.~that $\abs{\ga}<\sqrt{1-\eps_0^2}$. However, we also note that if $\g=u(\arg\ga)$, then $d(\g,g)=1-\abs{\ga}=\frac{\eps_0^2}{2}+O(\eps_0^4)$ (by considering the Taylor series of $\sqrt{1-x^2}$). Therefore, if we have the additional assumption that $\eps_0^2\approx\eps$ (which is sensible, since if $\eps$ is ``small'' then $\sqrt{\eps}$ will be ``large'', relatively speaking, and this is not incompatible with $\eps\ll\eps_0$), we get a bifurcated approach for {\em any} $g=u(\ga,\gb)\in\PU(2)$! Namely, in the case that $\abs{\ga}<\sqrt{1-\eps_0^2}$ we run our algorithm as specified above, and otherwise we give $u(\arg\ga)$ to \cite{RS16}'s algorithm. 
	\end{remark-non}
	
\section{Analysis of the algorithm}\label{sec:analysis}

	We begin the analysis by establishing the $T$-count and tightness of the approximation. In particular, $d(\g_1,\gd_1)<\eps$ and $d(\g_2,\gd_2)<\eps$ with factorization lengths each $(1+o(1))\log_2\frac{1}{\eps^3}$. By \lemref{lem:close approx}, $d(g,\gd_1\g\gd_2)<\nye{c}\eps$. Therefore, $d(g,\g_1\g\g_2)<(\tilde{c}+2)\eps$ (by the triangle inequality) and since $\g$ has a factorization of $T$-count approximately $\frac{1}{3}\log_2\frac{1}{\eps^3}$, this constitutes a factorization of an element in $g$'s neighborhood of $T$-count $\frac{7}{3}\log_2\frac{1}{\eps^3}$. 
	
	The efficiency of this algorithm---that is, that it runs in time $O\lpr{\poly\log\frac{1}{\eps}}$---is because we expect to halt when $2^k\eps^3\in O(1)$ (so only $k\approx\frac{1}{3}\log_2\frac{1}{\eps^3}$ calls are expected), and only call polynomially-many polynomial-time subroutines. The dominant subroutines are calls to Lenstra's algorithm, which as shown in \cite{Len} which runs in time polynomial in the size of the constraints for any fixed dimension $n$. Indeed, here we have only $m=6$ linear constraints (two per absolute value), and the largest value $a$ in the constraints is $p^k$, so the runtime is polynomial in $nm\log a\in\gT(k)$. 
	
	The reason we expect to halt when $2^k\eps^3\in O(1)$ is that $\G$ is viewed as an infinite 3-regular tree with each edge taking one further from the identity matrix's vertex represents an increase in the weight. The number of vertices of $T$-count exactly $k$ is $3\cdot2^{k-1}$, so there are $\gT(2^k)$ elements of $T$-count up to $k$. Expecting them to cover $\PU(2)$ (of total volume 1) well for large $k$, with a ball of volume some constant multiple $c$ of $\eps^3$, we therefore have $c2^k\eps^3\approx1$, i.e.~$k\approx\frac{1}{3}\log_2\frac{1}{\eps^3}$. As $T$-count is increasing in the quantity $k$ as in \eqref{eq:gamma def}, we expect to only have about $\frac{1}{3}\log_2\frac{1}{\eps^3}$-many iterations to achieve $T$-count $\frac{1}{3}\log_2\frac{1}{\eps^3}$. 
	
	When attempting to write elements of $\cO$ as a sum of two squares, we primarily rest on a belief, in the style of Cram\'{e}r's conjecture and a conjecture of Sardari \cite[($*$)]{Sar}, that sums of squares are dense in $\N$. Seeking to analogize \cite[($*$)]{Sar} in particular, we note that the operative aspect is that a dense cluster of lattice points will represent a sum of two squares, and that a point accomplishing this will be found quickly through Lenstra's algorithm. 
	
	The significance of this result is to accomplish a factorization in $\PU(2)$ shorter than the $3\log_2\frac{1}{\eps^3}$-factorization first demonstrated in \cite{PS}, precisely the desired outcome. 

\section{Implementation and an example}\label{sec:implementation}
	Visit \url{https://math.berkeley.edu/~zstier/ugthesis/HTLenstra.hs} for a Haskell implementation of code for the algorithm described in \secref{sec:continuous alg}, using many functions implemented in \cite{RS18}, and visit \url{https://math.berkeley.edu/~zstier/ugthesis/HTLenstra.nb} for a partial Mathematica implementation. 
	
	Let us now attempt to approximate
$$g=\frac{1}{3}\begin{pmatrix}1 & 2+2i \\ -2+2i & 1\end{pmatrix}$$
in $d$ to within $\eps=10^{-10}$ (times a small positive constant) using the Haskell implementation. Our heuristics provide an expectation of $T$-count near $\frac{7}{3}\log_210^{30}\approx232.5$. 

The algorithm described in \secref{sec:continuous alg} halts at $k=17$ and returns the approximation
$$\g=\frac{1}{2^9}\begin{pmatrix}-121+145\sqrt{2}+(123-192\sqrt{2})i & 103+78\sqrt{2}-\lpr{211+157\sqrt{2}}i \\ -103-78\sqrt{2}-\lpr{211+157\sqrt{2}}i & -121+145\sqrt{2}-(123-192\sqrt{2})i\end{pmatrix}.$$
This factors as 
\begin{align*}
	\g=&\  THTHTSHTHTSHTHTHTSHTHTSHTSHTSHTSHTSHTSHT \\
		&\ SHTSHTSHTHTSHTSHTSHTSHTSHTSHTSHTSHTSHTHT \\
		&\ SHTSHTSHSSSHH,
\end{align*}
with $T$-count 32---actually shorter than the expectation of $\frac{1}{3}\log_210^{30}\approx33.2$. We readily compute 
\begin{align*}
	\gt_1&\approx\phantom{-}1.477137\\
	\gt_2&\approx-0.421352
\end{align*}
and approximate the corresponding matrices $\gd_1=u(\gt_1)$ and $\gd_2=u(\gt_2)$ as 
\begin{align*}
	\g_1=&\ SHTHTSHTSHTSHTHTSHTHTHTHTHTSHTHTSHTSHTHT \\
		&\  SHTSHTSHTHTHTSHTSHTHTHTSHTHTHTSHTSHTSHTH \\
		&\  THTHTSHTSHTSHTSHTHTSHTHTSHTSHTHTSHTHTSHT \\
		&\  SHTHTSHTHTSHTHTHTSHTSHTSHTHTSHTSHTSHTSHT \\
		&\  HTHTHTHTHTSHTSHTSHTHTHTSHTHTSHTHTSHTSHTS \\
		&\  HTHTSHTHTHTHTSHTSHTSHTHTHTSHTSHTHTHTHTSH \\
		&\  THTHTHTSHTHTSHTSHTHSSSHHSSS,\\
	\g_2=&\ HTSHTHTSHTSHTSHTSHTHTHTSHTHTHTSHTHTHTSHT \\
		&\  HTHTHTHTHTSHTSHTHTSHTSHTHTHTSHTSHTHTHTSH \\
		&\  TSHTSHTSHTHTSHTHTHTSHTSHTHTSHTSHTSHTHTHT \\
		&\  HTSHTSHTHTHTHTSHTHTHTHTHTHTSHTHTSHTSHTSH \\
		&\  THTSHTSHTHTHTSHTHTSHTHTHTHTSHTHTSHTHTSHT \\
		&\  SHTHTSHTSHTHTHTSHTSHTSHTSHTHTHTSHTSHTHTH \\
		&\  THTHTSHTHTSHTSSSHH,
\end{align*}
both of $T$-counts 102, very near to $\log_210^{30}\approx99.6$. In total, this factorization gives $T$-count 236 (fewer than four more than the predicted value of $\frac{7}{3}\log_210^{30}$), and explicitly multiplying out the matrices gives precision with respect to $d$ within $\frac{4}{10^{21}}$, well within the guarantees of \lemref{lem:close approx} with $\nye{\eps}=\frac{1}{10^9}$ and $\eps_0=\frac{5}{6}$ (any $\nye{\eps}>\eps$ and $\eps_0<\frac{2\sqrt{2}}{3}$ suffice), which predicts an approximation to within $\lpr{2+\frac{2+10^{-9}}{\frac{5}{6}}}\frac{1}{10^{10}}\approx4.4\cdot\frac{1}{10^{10}}$. 

\section{Other gate sets}
	Consider the golden gate set of ``V-gates,''\footnote{So named due to the connection with the Roman numeral V.} with their notation preserved from \cite{Sar15}: 
	\begin{align*}
		s_1=\frac{1}{\sqrt{5}}\begin{pmatrix}1+2i \\ & 1-2i\end{pmatrix}, && s_2=\frac{1}{\sqrt{5}}\begin{pmatrix}1 & 2i \\ 2i & 1\end{pmatrix}, && s_3=\frac{1}{\sqrt{5}}\begin{pmatrix}1 & 2 \\ -2 & 1\end{pmatrix},
	\end{align*}
	with $\cS=\lcr{s_1^{\pm1},s_2^{\pm1},s_3^{\pm1}}$ and $\G=\<s_1,s_2,s_3\>$. This case is actually even simpler than that of the Clifford+$T$ gates, requires just a 1-dimensional variant of Lenstra's algorithm, and is described in full, and implemented, in \cite[\S2.2]{Sti}. 
	
	The case of general golden gate sets will be the subject of a future paper. 

\section*{Acknowledgements}
Part of this work appears in my senior thesis at Princeton University. I am grateful to my advisor Peter Sarnak for his support and guidance throughout. I am supported by a National Science Foundation Graduate Research Fellowship, grant number DGE-1752814.

\end{document}